\documentclass[11pt]{article}
\usepackage{amsmath,amssymb,amscd,amsthm}
\usepackage{subdepth}
\newtheorem{theorem}                   {Theorem}
\newtheorem{thm}             [theorem] {Theorem}
\newtheorem{lemma}           [theorem] {Lemma}

\newtheorem{claim}           [theorem] {Claim}

\def\dcup{\,\dot\cup\,}
\def\ex{\textup{ex}}

\begin{document}

\title{Multicoloured Ramsey numbers\\[0ex]
of the path of length four\\[1ex]}
\author{Henry Liu%
\thanks{School of Mathematics, Sun Yat-sen University, Guangzhou 510275, China. Email: {\tt  liaozhx5@mail.sysu.edu.cn} (corresponding author)}
\and Bojan Mohar%
\thanks{Department of Mathematics, Simon Fraser University, Burnaby, BC, V5A 1S6, Canada. On leave from IMFM, Department of Mathematics, University of Ljubljana, 1000 Ljubljana, Slovenia. Email: {\tt mohar@sfu.ca}}
\and Yongtang Shi%
\thanks{Center for Combinatorics and LPMC, Nankai University, Tianjin 300071, China. Email: {\tt shi@nankai.edu.cn}}\\[1ex]
}
\date{14 August 2021}
\maketitle

\begin{abstract}
Let $P_t$ denote the path on $t$ vertices. The \emph{$r$-coloured Ramsey number} of $P_t$, denoted by\, $R_r(P_t)$,\, is\, the\, minimum\, integer\, $n$\, such\, that\, whenever\, the\, complete\, graph\, on\, $n$ vertices is given an $r$-edge-colouring, there exists a monochromatic copy of $P_t$. In this note, we determine $R_r(P_5)$, which is approximately $3r$.\\



\end{abstract}

\section{Introduction}

All graphs in this paper are finite, undirected, and have no multiple edges or loops. For any undefined terms in graph theory, we refer to the book by Bollob\'as \cite{Bol98}.

Let $K_t,P_t$ and $C_t$ denote the complete graph (or clique), 
path and cycle on $t$ vertices. For graphs $G$ and $H$, we denote the graph which is the disjoint union of a copy of $G$ and a copy of $H$ by $G\dcup H$, and the graph with $a$ disjoint copies of $G$ by $aG$. The \emph{graph union} of $G$ and $H$ is $G\cup H=(V(G)\cup V(H), E(G)\cup E(H))$. 

An \emph{$r$-edge-colouring} of a graph $G$, or \emph{$r$-colouring} for simplicity, is a function $f:E(G)\to\{1,\dots,r\}$. The members of the set $\{1,\dots,r\}$ can be thought of as a set of $r$ colours. 
The sets $f^{-1}(i)$ for $1\le i\le r$ are the \emph{colour classes} of the $r$-colouring $f$.

Given graphs $H_1,\dots,H_r$, the \emph{$r$-coloured Ramsey number} $R(H_1,\dots,H_r)$ is the minimum integer $n$ such that, whenever we have an $r$-edge-colouring of $K_n$, there exists a monochromatic copy of $H_i$, for some $i$. When $H_1=\cdots=H_r=H$, we write $R_r(H)$ for $R(H,\dots,H)$. Ramsey's classical result \cite{Ram30} states that all Ramsey numbers $R(H_1,\dots,H_r)$ exist. When all the $H_i$ are cliques, determining the Ramsey numbers exactly is a notoriously challenging problem, and only a few values of $R(K_s,K_t)$, as well as $R(K_3,K_3,K_3)$, are known. Another well studied case is when all the $H_i$ are cycles, and in this case, the values of $R(C_s,C_t)$ are completely determined, while there are many interesting results and open questions for $R(C_q,C_s,C_t)$. For more information about Ramsey numbers, we refer the reader to the survey paper of Radziszowski \cite{Rad21}.

For the case when the $H_i$ are paths, Gerencs\'er and Gy\'arf\'as \cite{GG67} proved that $R(P_s,P_t)=t+\lfloor\frac{s}{2}\rfloor-1$ for $t\ge s\ge 2$. For three colours, Faudree and Schelp \cite{FS75} conjectured that $R(P_t,P_t,P_t)=2t-2+(t$ mod $2)$ for all $t$. This conjecture has been verified for all sufficiently large $t$ by Gy\'arf\'as et al.~\cite{GRSS07}. There are many known exact values for $R(P_q,P_s,P_t)$ when $q$ and $s$ are small. For more colours, since it is well known that the edge-chromatic number of $K_n$ is $n-1+(n$ mod $2)$, we have $R_r(P_3)=r+1+(r$ mod $2)$. For $P_4$, we have
\[
R_r(P_4)=
\left\{
\begin{array}{l@{\quad\quad}l}
2r+1 & \textup{if }r\equiv 0,2\textup{ (mod 3), }r\neq 3,\\[1ex]
2r+2 & \textup{if }r\equiv 1\textup{ (mod 3)},\\[1ex]
6 & \textup{if }r=3.
\end{array}
\right.
\]

The cases $r\equiv 1,2$ (mod 3) were proved by Irving \cite{Irv74}, and he also remarked that $R_3(P_4)=6$. The case $r\equiv 0$ (mod 3) with $r$ not a power of $3$ was proved by Lindstr\"om \cite{Lin83}. The case where $r$ is a power of $3$ was proved by Bierbrauer \cite{Bie86}.

Here, we shall determine the Ramsey numbers $R_r(P_5)$ exactly, as follows.

\begin{thm}\label{RrP5thm}
Let $r\ge 1$. Then 
\[
R_r(P_5)=
\left\{
\begin{array}{l@{\quad\quad}l}
3r+1 & \textup{\emph{if} }r\equiv 0\textup{ (mod 4)\emph{,} $r\neq 4$\emph{,}}\\[1ex]
3r+2 & \textup{\emph{if} }r\equiv 1\textup{ (mod 4)\emph{,}}\\[1ex]
3r & \textup{\emph{if} }r\equiv 2,3\textup{ (mod 4)\emph{,}}\\[1ex]
11 & \textup{\emph{if} }r=4.
\end{array}
\right.
\]
\end{thm}

We shall proceed as follows. In Section \ref{toolssec}, we gather various results from Ramsey theory, Tur\'an theory, and design theory. In Section \ref{thm1sec}, we prove Theorem \ref{RrP5thm}.


\section{Tools}\label{toolssec}

We first observe that for $r\ge 2$ and $t\ge 4$, we have
\begin{equation}
R_{r-1}(P_t)<R_r(P_t).\label{Rincr}
\end{equation}
Indeed, both terms in (\ref{Rincr}) exist. Let $R= R_{r-1}(P_t)$. Then, there exists an $(r-1)$-colouring of $K_{R-1}$ which does not contain a monochromatic copy of $P_t$. By adding new vertex $u$ to form $K_R$ and colouring all edges incident to $u$ with a new colour, we get an $r$-colouring of $K_R$ with no monochromatic copy of $P_t$ (since $t\ge 4$). Hence, $R_r(P_t)>R=R_{r-1}(P_t)$.

Next, we recall that for a graph $H$ and $n\in\mathbb N$, the \emph{Tur\'an function} for $H$, denoted by ex$(n,H)$, is the maximum possible number of edges in a $H$-free graph (i.e., not containing a copy of $H$ as a subgraph) on $n$ vertices. Any $H$-free graph on $n$ vertices and attaining ex$(n,H)$ edges is said to be \emph{extremal}. For any path $P_t$, the Tur\'an function ex$(n,P_t)$, as well as the corresponding extremal graphs, were completely determined by Faudree and Schelp \cite{FS75}. In order to attain ex$(n,P_t)$, we can take the graph on $n$ vertices containing as many disjoint $K_{t-1}$ cliques as possible, i.e., $\lfloor\frac{n}{t-1}\rfloor$ cliques, and a smaller clique on the remaining vertices. For odd $t$, this graph is the unique extremal graph, and for even $t\ge 4$ and certain values of $n$, there are other extremal graphs. Their result for $P_5$ is the following.

\begin{thm}\label{FSthm}\textup{\cite{FS75}}\,
Let $n=4a+b$, where $a\ge 0$ and $0\le b\le 3$. We have
\[
\textup{ex}(n, P_5)=6a+\frac{b(b-1)}{2}.
\]
Moreover, the unique extremal graph is $aK_4\dcup K_b$.
\end{thm}

Now, we recall some notions and results from design theory. Let $v\ge k\ge t$ and $\lambda$ be positive integers. Let $V$ be a set of $v$ points, and $\mathcal B$ be a family of $k$-element subsets of $V$, called \emph{blocks}. The pair $(V,\mathcal B)$ is a  \emph{Steiner system} $S(t,k,v)$ if any $t$ points of $V$ are contained in exactly one block of $\mathcal B$. If ``exactly one'' is replaced by ``at least one'' and ``at most one'', then $(V,\mathcal B)$ is a \emph{$(t,k,v)$-covering design} and a \emph{$(t,k,v)$-packing design}, respectively. A $(t,k,v)$-covering design (resp.~$(t,k,v)$-packing design) $(V, \mathcal B)$ is \emph{optimal} if $|\mathcal B|$ attains the minimum (resp.~maximum) possible value. The pair $(V,\mathcal B)$ is a \emph{balanced incomplete block design}, or \emph{BIBD}, if any two points of $V$ are contained in exactly $\lambda$ blocks. Thus, a Steiner system $S(2,k,v)$ is precisely a BIBD with $\lambda=1$, and we denote such a design by $B(k,v)$. It is well-known that a necessary condition for the existence of a $B(k,v)$ is $v\equiv 1$ or $k$ (mod $k(k-1))$. For a $(2,k,v)$-packing design $(V,\mathcal B)$, the \emph{leave graph} $L$ is the graph with vertex set $V(L)=V$, and edge set $E(L)=\{xy: x,y\in V$ and $\{x,y\}\not\subset B$ for every $B\in\mathcal B\}$.
%

For any of the above designs, a \emph{parallel class} is a set of blocks of $\mathcal B$ that form a partition of $V$. If $\mathcal B$ can be partitioned into parallel classes, then the design is \emph{resolvable}. Clearly for a design to be resolvable, we must necessarily have $v\equiv 0$ \textup{(mod} $k)$.

Here, we are interested in designs with $k=4$, $t=2$ and $\lambda=1$. We refer the interested reader to a survey by Reid and Rosa \cite{RR10} for more information about these designs. For BIBDs, we have the following result of Hanani et al.~\cite{HRW72}

\begin{thm}\label{RBIBDthm}\textup{\cite{HRW72}}\,
For every $v\equiv 4 \textup{ (mod }12)$, there exists a resolvable BIBD $B(4,v)$. The number of parallel classes is $\frac{v-1}{3}$.
\end{thm}

%

For $(2,4,v)$-covering designs, by combining results of Lamken et al.~\cite{LMR98}, and Abel et al.~\cite{AABBG07}, we have the following result.

\begin{thm}\label{covthm}\textup{\cite{AABBG07,LMR98}}\,
For $v\equiv 0,8$\textup{ (mod }$12)$ with $v\ne 12$, there exists an optimal $(2,4,v)$-covering design which is resolvable, except possibly for $v\in\{108, 116, 132, 156, 204, 212\}$. The number of parallel classes is
\[
\left\{
\begin{array}{l@{\quad\quad}l}
\displaystyle \frac{v}{3} & \textup{\emph{if} }v\equiv 0\textup{ (mod 12)\emph{,}}\\[2ex]
\displaystyle \frac{v+1}{3} & \textup{\emph{if} }v\equiv 8\textup{ (mod 12).}
\end{array}
\right.
\]
\end{thm}

For $(2,4,v)$-packing designs, the following result was proved by Ge et al.~\cite{GLLS05}. Several cases of the result were proved by various authors. See \cite{GLLS05} for the references therein.

\begin{thm}\label{RGDDthm}\textup{\cite{GLLS05}}\,
\begin{enumerate}
\item[(a)] For $v\equiv 0$\textup{ (mod }$12)$ with $v\neq 12$, there exists an optimal $(2,4,v)$-packing design which is resolvable. The number of parallel classes is $\frac{v-3}{3}$, and the leave graph is $L=\frac{v}{3}K_3$.
\item[(b)] For $v\equiv 8$\textup{ (mod }$12)$ with $v\neq 8,20$, there exists an optimal $(2,4,v)$-packing design which is resolvable, except possibly for $v\in\{68, 92, 104, 140, 164, 188, 200, 236,260, 284, 356,$ $368, 404, 428, 476, 500,668, 692\}$. The number of parallel classes is $\frac{v-2}{3}$, and the leave graph is $L=\frac{v}{2}K_2$.
\end{enumerate}
\end{thm}

%


\section{Ramsey numbers of the path of length four}\label{thm1sec}

In this section, we prove Theorem \ref{RrP5thm}. First, we prove the required upper bound for $R_r(P_5)$, when $r\neq 4$.

\begin{lemma}\label{RrP5lem1}
For $r\ge 1$, we have
\[
R_r(P_5)\le\left\{
\begin{array}{l@{\quad\quad}l}
\displaystyle 3r+1 & \textup{\emph{if} }r\equiv 0\textup{ (mod 4)\emph{,}}\\[1ex]
\displaystyle 3r+2 & \textup{\emph{if} }r\equiv 1\textup{ (mod 4)\emph{,}}\\[1ex]
\displaystyle 3r & \textup{\emph{if} }r\equiv 2,3\textup{ (mod 4).}
\end{array}
\right.
\]
\end{lemma}

\begin{proof}
Let $n=4a+b$, where $a\ge 0$ and $0\le b\le 3$. By Theorem \ref{FSthm}, we have $\ex(n,P_5)=6a+\frac{1}{2}b(b-1)=\frac{3}{2}n+\frac{1}{2}b^2-2b$, and the unique extremal graph is $aK_4\dcup K_b$. Now suppose that we have an $r$-colouring of $K_n$. Then the most frequent colour, say red, has at least $\lceil\frac{1}{r}{n\choose 2}\rceil$ edges. If $r\equiv 0$ (mod $4)$ and $n=3r+1$, or $r\equiv 1$ (mod $4)$ and $n=3r+2$, then $b=1$, so that $\ex(n,P_5)=\frac{3}{2}n-\frac{3}{2}$. We have
\[
\bigg\lceil\frac{1}{r}{n\choose 2}\bigg\rceil\ge \frac{3}{2}n>\ex(n,P_5)+1.
\]

If $r\equiv 2$ (mod $4)$ and $n=3r$, then $b=2$, so that $\ex(n,P_5)=\frac{3}{2}n-2$. We have
\[
\bigg\lceil\frac{1}{r}{n\choose 2}\bigg\rceil = \bigg\lceil\frac{3(n-1)}{2}\bigg\rceil= \frac{3(n-1)+1}{2}=\ex(n,P_5)+1.
\]

In all three cases, we have a red $P_5$.

If $r\equiv 3$ (mod $4)$ and $n=3r$, then $b=1$. We have
\[
\bigg\lceil\frac{1}{r}{n\choose 2}\bigg\rceil=\frac{3(n-1)}{2}=\ex(n,P_5).
\]
Suppose that there is no monochromatic copy of $P_5$. Then, every colour class has exactly $\lceil\frac{1}{r}{n\choose 2}\rceil=\textup{ex}(n,P_5)$ edges, and moreover, must induce the unique extremal graph $aK_4\dcup K_1$. Let $u$ be the vertex of $K_n$ which corresponds to $K_1$ in the red $aK_4\dcup K_1$. Then for every other colour, the number of edges incident to $u$ is three or zero, so that there are at most $3(r-1)$ edges at $u$. This contradicts that $u$ has degree $3r-1$ in $K_n$.
\end{proof}

Next, we prove the matching lower bound for $R_r(P_5)$, when $r\neq 4$.

\begin{lemma}\label{RrP5lem2}
For $r\ge 1$ with $r\neq 4$, we have
\[
R_r(P_5)\ge\left\{
\begin{array}{l@{\quad\quad}l}
\displaystyle 3r+1 & \textup{\emph{if} }r\equiv 0\textup{ (mod 4)\emph{,}}\\[1ex]
\displaystyle 3r+2 & \textup{\emph{if} }r\equiv 1\textup{ (mod 4)\emph{,}}\\[1ex]
\displaystyle 3r & \textup{\emph{if} }r\equiv 2,3\textup{ (mod 4)}.
\end{array}
\right.
\]
\end{lemma}

\begin{proof}
We first note that the case $r\equiv 2$ (mod $4)$ follows from the case $r\equiv 1$ (mod $4)$, since by (\ref{Rincr}), we have $R_r(P_5)\ge R_{r-1}(P_5)+1\ge 3(r-1)+2+1=3r$ for $r\equiv 2$ (mod $4)$.

Now, define
\[
g(r)=
\left\{
\begin{array}{l@{\quad\quad}l}
\displaystyle 3r\equiv 0\textup{ (mod }12) & \textup{if }r\equiv 0\textup{ (mod 4), }r\neq 4,\\[1ex]
\displaystyle 3r+1 \equiv 4\textup{ (mod }12) & \textup{if }r\equiv 1\textup{ (mod 4),}\\[1ex]
\displaystyle 3r-1 \equiv 8\textup{ (mod }12) & \textup{if }r\equiv 3\textup{ (mod 4)}.
\end{array}
\right.
\]

For $r\equiv 1$ (mod $4)$, by Theorem \ref{RBIBDthm}, there exists a resolvable BIBD $B(4,g(r))$. For $r\equiv 0,3$ (mod $4)$, by Theorem \ref{covthm}, there exists an optimal $(2,4,g(r))$-covering design which is resolvable, except for $g(r)\in\{108, 116, 132, 156, 204, 212\}$. In each case, we obtain an $r$-colouring of $K_{g(r)}$, where an edge $xy$ is given colour $i$ if $x$ and $y$ are contained in a block in the $i$th parallel class. The number of parallel classes is
\[
\left\{
\begin{array}{l@{\quad\quad}l}
\displaystyle \frac{3r}{3}=r & \textup{if }r\equiv 0\textup{ (mod 4), }r\neq 4,\\[2ex]
\displaystyle \frac{(3r+1)-1}{3}=r & \textup{if }r\equiv 1\textup{ (mod 4),}\\[2ex]
\displaystyle \frac{(3r-1)+1}{3}=r & \textup{if }r\equiv 3\textup{ (mod 4)}.
\end{array}
\right.
\]
We have an $r$-colouring of $K_{g(r)}$ which does not contain a monochromatic component with more than four vertices, and thus does not contain a monochromatic copy of $P_5$.

Now let $r\equiv 0$ (mod $4)$ and $g(r)=3r\in\{108, 132, 156, 204\}$. By Theorem \ref{RGDDthm}(a), there exists an optimal $(2, 4, g(r))$-packing design which is resolvable. The number of parallel classes is $\frac{3r-3}{3}=r-1$, and the leave graph is $L=rK_3$. Similarly, for $r\equiv 3$ (mod $4)$ and $g(r)=3r-1\in\{116, 212\}$, by Theorem \ref{RGDDthm}(b), there exists an optimal $(2, 4, g(r))$-packing design which is resolvable. Note that $116$ and $212$ are not in the list of 18 exceptional values in Theorem \ref{RGDDthm}(b). The number of parallel classes is $\frac{(3r-1)-2}{3}=r-1$, and the leave graph is $L=\frac{3r-1}{2}K_2$. In both cases, we obtain an $r$-colouring of $K_{g(r)}$, where an edge $xy$ is given colour $i$ if $x$ and $y$ are contained in a block in the $i$th parallel class, for $1\le i\le r-1$; and colour $r$ if $xy\in E(L)$. Then, all monochromatic components in this $r$-colouring are $K_4$, $K_3$ or $K_2$, so there is no monochromatic copy of $P_5$.

This means that for every $r\neq 4$, we have $R_r(P_5)\ge g(r)+1$.
\end{proof}

To complete the proof of Theorem \ref{RrP5thm}, it remains to compute $R_4(P_5)$.

\begin{lemma}\label{RrP5lem3}
$R_4(P_5)=11$.
\end{lemma}
\begin{proof}
We first obtain the lower bound $R_4(P_5)\ge 11$. Let $x_1,x_2,y_1,\dots,y_4,z_1,\dots,z_4$ be $10$ vertices, and $G_1,\dots,G_4$ to be the graphs consisting of disjoint cliques on the following sets.\\[1ex]
\indent $G_1$: $\{x_1,x_2\}$, $\{y_1,\dots,y_4\}$, $\{z_1,\dots,z_4\}$;\\[1ex]
\indent $G_2$: $\{x_1,y_1,y_2,z_3\}$, $\{x_2,z_1,z_2,y_3\}$, $\{y_4,z_4\}$;\\[1ex]
\indent $G_3$: $\{x_1,z_1,z_2,y_4\}$, $\{x_2,y_1,y_2,z_4\}$, $\{z_3,y_3\}$;\\[1ex]
\indent $G_4$: $\{x_1,y_3,z_4\}$, $\{x_2,y_4,z_3\}$, $\{y_1,y_2,z_1,z_2\}$.\\[1ex]
\indent It is easy to verify that $G_1\cup G_2\cup G_3\cup G_4=K_{10}$. We obtain a $4$-colouring of $K_{10}$ where an edge $xy$ is given colour $i$ if $i$ satisfies $xy\in E(G_i)$. This $4$-colouring does not contain a monochromatic copy of $P_5$, and hence $R_4(P_5)\ge 11$.

Now we prove the matching upper bound $R_4(P_5)\le 11$. Let $K_4^-$ denote the graph obtained by deleting an edge from $K_4$.
\begin{claim}\label{R4P5clm}
%
If $G$ is a $P_5$-free graph with $11$ vertices and $14$ edges, then $G=K_4\dcup K_4\dcup P_3$ or  $G=K_4\dcup K_4^-\dcup K_3$.
\end{claim}
\begin{proof}
Let $F$ be a component of $G$ with $s$ vertices. Since $G$ is $P_5$-free, if $F$ contains a cycle, then the longest cycle of $F$ has length $3$ or $4$. If the former, then $F$ is a $C_3$ with possibly some pendent edges attached to one vertex, and $e(F)=s$. If the latter, then $F=C_4,K_4^-$ or $K_4$, and $e(F)=s,s+1$ or $s+2$ respectively. Otherwise, $F$ is a $P_5$-free tree, and $e(F)=s-1$. Since $|V(G)|=11$, at most two components of $G$ can be $K_4^-$ or $K_4$. Also, since $e(G)=|V(G)|+3$, this means that exactly two components are either both $K_4$, or one is $K_4^-$ and the other is $K_4$. We can then easily see that $G=K_4\dcup K_4\dcup P_3$ or  $G=K_4\dcup K_4^-\dcup K_3$.
\end{proof}

Suppose that we have a $4$-colouring of $K_{11}$ with vertex set $V$, which does not contain a monochromatic copy of $P_5$. Let $G_1,\dots, G_4$ be the four graphs on $V$ induced by the colour classes, and assume that $e(G_1)\ge\cdots\ge e(G_4)$. By Theorem \ref{FSthm}, we have ex$(11,P_5)=15$, with the unique extremal graph $K_4\dcup K_4\dcup K_3$. Hence, $e(G_1)\le 15$. If $e(G_1)=15$, then $G_1=K_4\dcup K_4\dcup K_3$, and $e(G_2)\ge \big\lceil\frac{e(K_{11})-15}{3}\big\rceil=\big\lceil\frac{55-15}{3}\big\rceil=14$. By Claim \ref{R4P5clm}, $G_2$ contains a copy of $K_4$, which is a contradiction since the complement of $G_1$ does not contain a copy of $K_4$. Otherwise, we have $e(G_1)=e(G_2)=e(G_3)=14$ and $e(G_4)=13$. Again by Claim \ref{R4P5clm}, we have $G_1=K_4\dcup K_4\dcup P_3$ or $K_4\dcup K_4^-\dcup K_3$, and $G_2$ and $G_3$ both contain a copy of $K_4$. We see that $G_1\cup G_2$ must contain a copy of $K_4\dcup K_4\dcup K_3$, so that the complement of $G_1\cup G_2$ does not contain a copy of $K_4$. This contradicts that $G_3$ contains a copy of $K_4$.

This completes the proof of Lemma \ref{RrP5lem3}.
\end{proof}

By Lemmas \ref{RrP5lem1}, \ref{RrP5lem2} and \ref{RrP5lem3}, Theorem \ref{RrP5thm} is proved.


%


\section*{Acknowledgements}
Henry Liu is partially supported by the Startup Fund of One Hundred Talent Program of\, SYSU,\, and\, National\, Natural\, Science\, Foundation\, of\, China\, (No.~11931002).\, Bojan Mohar is partially supported by the NSERC Discovery Grant R611450 (Canada), and by the Research Project J1-2452 of ARRS (Slovenia). Yongtang Shi is partially supported by National Natural Science Foundation of China (Nos.~11771221 and 11922112), and Natural Science Foundation of Tianjin (Nos.~20JCJQJC00090 and
20JCZDJC00840).

\end{document}